\documentclass{amsart}
\usepackage{amssymb,amsmath,amsthm,hyperref}
\newtheorem{theorem}{Theorem}[section]
\newtheorem{lemma}[theorem]{Lemma}
\newtheorem{cor}[theorem]{Corollary}

\theoremstyle{definition}
\newtheorem{definition}[theorem]{Definition}
\newtheorem{example}[theorem]{Example}

\theoremstyle{remark}

\numberwithin{equation}{section}

\newcommand{\spt}{\mbox{\rm spt}}

\DeclareSymbolFont{AMSb}{U}{msb}{m}{n}
\DeclareMathSymbol{\Z}{\mathalpha}{AMSb}{"5A}

\DeclareMathSymbol{\nmid}{\mathrel}{AMSb}{"2D}
\DeclareSymbolFont{AMSb}{U}{msb}{m}{n}
\DeclareMathSymbol{\C}{\mathalpha}{AMSb}{"43}
\DeclareMathSymbol{\F}{\mathalpha}{AMSb}{"46}
\DeclareMathSymbol{\N}{\mathalpha}{AMSb}{"4E}
\DeclareMathSymbol{\Q}{\mathalpha}{AMSb}{"51}
\DeclareMathSymbol{\R}{\mathalpha}{AMSb}{"52}
\DeclareMathSymbol{\Z}{\mathalpha}{AMSb}{"5A}

\allowdisplaybreaks

\begin{document}
\newcommand{\beqs}{\begin{equation*}}
\newcommand{\eeqs}{\end{equation*}}
\newcommand{\beq}{\begin{equation}}
\newcommand{\eeq}{\end{equation}}
\newcommand\mylabel[1]{\label{#1}}
\newcommand\eqn[1]{(\ref{eq:#1})}
\newcommand\thm[1]{\ref{thm:#1}}
\newcommand\lem[1]{\ref{lem:#1}}
\newcommand\propo[1]{\ref{propo:#1}}
\newcommand\corol[1]{\ref{cor:#1}}
\newcommand\sect[1]{\ref{sec:#1}}

\title[Higher Order spt-Functions]{
Higher order spt-functions}

\author{F.~G.~Garvan}
\address{Department of Mathematics, University of Florida, Gainesville,
Florida 32611-8105}
\email{fgarvan@ufl.edu}          
\thanks{The author was supported in part by NSA Grant H98230-09-1-0051.
Most of the research for this paper was done while the author was visiting
Mike Hirschhorn at the University of New South Wales in June 2010.
The first version of this paper was submitted to the arXiv on August 6, 2010.
}



\date{October 30, 2010}  

\dedicatory{Dedicated to the memory of 
A.O.L. (Oliver) Atkin}


\begin{abstract}
Andrews' spt-function can be written as the difference between
the second symmetrized crank and rank moment functions. Using the
machinery of Bailey pairs a combinatorial interpretation is given for
the difference between higher order symmetrized crank and rank moment
functions. This implies an inequality between crank and rank
moments that was only known previously for sufficiently large $n$ 
and fixed order.
This combinatorial interpretation is in terms of a weighted sum
of partitions. A number of congruences for higher order spt-functions
are derived.
\end{abstract}

\maketitle

\section{Introduction} \label{sec:intro}

Andrews \cite{An08b} defined the function $\spt(n)$ as the number of 
smallest parts in the partitions of $n$. He related this
function to the second rank moment. He also proved some surprising congruences
mod $5$, $7$ and $13$. Namely, he showed that
\begin{equation}
\spt(n) = n p(n) - \frac{1}{2} N_2(n),
\mylabel{eq:sptid}
\end{equation}
where $N_2(n)$ is the second rank moment function and $p(n)$ is the number of 
partitions of $n$, and he proved that
\begin{align*}
\spt(5n+4) &\equiv 0 \pmod{5},\\
\spt(7n+5) &\equiv 0 \pmod{7},\\
\spt(13n+6) &\equiv 0 \pmod{13}.   
\end{align*}
As noted in \cite{Ga10a}, \eqn{sptid} can be rewritten as
$$
\spt(n) = \frac{1}{2}(M_2(n) - N_2(n)),
$$
where $M_2(n)$ is the second crank moment function. 
Rank and crank moments were introduced by  A.~O.~L.~Atkin and the author
\cite{At-Ga}. Bringmann \cite{Br08} studied analytic, asymptotic
and congruence 
properties of the generating function for the second rank moment as
a quasi-weak Maass form. Further congruence properties of Andrews'
spt-function were found by the author \cite{Ga10a}, Folsom and Ono \cite{Fo-On}
and Ono \cite{On10}.
In \cite{Ga10a}
it was conjectured that
\begin{equation}
M_{2k}(n) > N_{2k}(n),
\mylabel{eq:MNineq}
\end{equation}
for all $k \ge 1$ and $n\ge 1$. 
Here $M_{2k}(n)$ and $N_{2k}(n)$ are the $2k$-th crank and $2k$-th rank moment
functions.
For each fixed $k$, the inequality was proved for sufficiently 
large $n$ by  Bringmann, Mahlburg and Rhoades \cite{Br-Ma-Rh},
who determined the asymptotic behaviour for the difference
$M_{2k}(n)-N_{2k}(n)$ (see Section \sect{remarks}).
The first few cases of the conjecture were previously proved by
Bringmann and Mahlburg \cite{Br-Ma}.
In this paper we
prove the inequality unconditionally for all $n$ and $k$ by
finding a combinatorial interpretation for  the difference between
symmetrized crank and rank moments.  
Analytic and arithmetic
properties of higher order rank moments were studied by
Bringmann, Lovejoy and Osburn \cite{Br-Lo-Os} and by
Bringmann, the author and Mahlburg \cite{Br-Ga-Ma}.

Andrews \cite{An07a} defined the $k$-th
symmetrized rank function by
$$
\eta_k(n) = \sum_{m=-n}^n \binom{m+\lfloor\frac{k-1}{2}\rfloor}{k} N(m,n),
$$
where $N(m,n)$ is the number of partitions of $n$ with rank $m$.
Andrews gave a new interpretation of the symmetrized rank function in terms
of Durfee symbols. 
As a natural analog to the symmetrized rank function we 
define the 
$k$-th
symmetrized crank function by
$$
\mu_k(n) = \sum_{m=-n}^n \binom{m+\lfloor\frac{k-1}{2}\rfloor}{k} M(m,n),
$$
where $M(m,n)$ is number of partitions of $n$ with crank $m$, for $n\ne1$.
For $n=1$
we define
$$
M(-1,1)=1, M(0,1)=-1, M(1,1)=1,\quad\mbox{and otherwise $M(m,1)=0$.}
$$
One of our main results is the following identity

\begin{align}
&\sum_{n=1}^\infty \left( \mu_{2k}(n) - \eta_{2k}(n)\right) q^n 
\mylabel{eq:mainid}\\
&=
\sum_{n_k \ge n_{k-1} \ge \cdots \ge n_1 \ge 1}
\frac{q^{n_1 + n_2 + \cdots + n_k}}
{(1-q^{n_k})^2 (1-q^{n_{k-1}})^2 \cdots (1-q^{n_1})^2(q^{n_1+1};q)_\infty}.
\nonumber
\end{align}
When $k=1$ this result reduces to \eqn{sptid}. 
In equation \eqn{mainid} and throughout this paper we use 
standard $q$-notation \cite{Ga-Ra-book}. We compare equation \eqn{mainid}
with the identity
\beq
\sum_{n=1}^\infty \mu_{2k}(n) q^n = 
\frac{1}{(q)_\infty}
\sum_{n_k \ge n_{k-1} \ge \cdots \ge n_1 \ge 1}
\frac{q^{n_1 + n_2 + \cdots + n_k}}
{(1-q^{n_k})^2 (1-q^{n_{k-1}})^2 \cdots (1-q^{n_1})^2},
\mylabel{eq:Mukidintro}
\eeq
which is proved in Section \sect{bailey}. Some remarks about this
identity are also given in Section \sect{remarks}.

In Section \sect{symM} we show that many of Andrews' results \cite{An07a}
for symmetrized rank moments can be extended to symmetrized crank moments.
In Section \sect{bailey} we prove a general result for Bailey pairs
from which our main identity \eqn{mainid} follows.
In Section \sect{MNineq},
we use an analog of Stirling numbers of the second kind to show how ordinary 
moments
can be expressed in terms of symmetrized moments and how our main identity
implies the inequality \eqn{MNineq}. For each $k\ge1$, we are able to define
a higher-order spt function $\spt_k(n)$ so that
$$
\spt_k(n) = \mu_{2k}(n) - \eta_{2k}(n),                     
$$
for all $k\ge1$ and $n\ge1$. In Section \sect{sptk}
we give the  combinatorial definition of $\spt_k(n)$
is in terms of a weighted sum over the partitions of $n$. We note
that when $k=1$, $\spt_k(n)$ coincides with Andrews' spt-function.

In Section \sect{congsptk} we prove a number of congruences for the higher order
spt-functions. In Section \sect{remarks} we make
some concluding remarks and close the paper with a table of
$\spt_k(n)$ for small $n$ and $k$.


\section{Symmetrized crank moments}\label{sec:symM}

In this section we collect some results for symmetrized crank moments.
Many of Andrews' results and proofs for symmetrized rank moments have analogs
for symmetrized crank moments; thus we
omit some details.

\begin{align*}
C(z,q) &:= \sum_{n=1}^\infty \sum_{m=-n}^n M(m,n) z^m q^n \\
&= \frac{1}{(q)_\infty}
   \left(1 + \sum_{n=1}^\infty \frac{ (1-z)(1-z^{-1}) (-1)^n q^{n(n+1)/2}(1 + q^n)}
                                    {(1-zq^n)(1-z^{-1}q^n)}\right)\qquad\\
& \hspace{3in}\mbox{(by \cite[Eq.(7.15), p.70]{Ga88a})}\\
&= \frac{1}{(q)_\infty}
   \left(1 + \sum_{n=1}^\infty (-1)^n q^{n(n+1)/2} 
    \left( \frac{1-z}{1-zq^n} + \frac{1-z^{-1}}{1-z^{-1}q^n}\right)\right)\\
&= \frac{(1-z)}
        {(q)_\infty}
    \sum_{n=-\infty}^\infty \frac{(-1)^n q^{n(n+1)/2}}{1-zq^n},
\end{align*}
and
$$
C^{(j)}(z,q) = \frac{-j!}{(q)_\infty} \sum_{\substack{n=-\infty\\n\ne0}}^\infty
                                      \frac{(-1)^nq^{n(n-1)/2 + jn}(1-q^n)}
                                           {(1-zq^n)^{j+1}},
$$
for $j\ge0$.

By \cite[Theorem 1]{An07a} we know that $\eta_k(n)=0$ if $k$ is odd.
In a similar fashion we find that $\mu_k(n)=0$ if $k$ is odd.

We will need
\begin{theorem}[Andrews \cite{An07a}]
\begin{equation}
\sum_{n=1}^\infty \eta_{2k}(n) q^n =
\frac{1}{(q)_\infty} \sum_{n=1}^\infty (-1)^{n-1} q^{n(3n-1)/2 + kn}
\frac{(1+q^n)}
     {(1-q^n)^{2k}}
=\frac{1}{(q)_\infty}
   \sum_{\substack{n=-\infty\\n\ne0}}^\infty
                                      \frac{(-1)^{n-1}q^{n(3n+1)/2 + kn}}
                                           {(1-q^n)^{2k}}.
\mylabel{eq:symNid}
\end{equation}
\end{theorem}

This theorem has a crank analog.                                     

\begin{theorem}
\begin{equation}
\sum_{n=1}^\infty \mu_{2k}(n) q^n = 
\frac{1}{(q)_\infty} \sum_{n=1}^\infty (-1)^{n-1} q^{n(n-1)/2 + kn} 
\frac{(1+q^n)}
     {(1-q^n)^{2k}}
=\frac{1}{(q)_\infty}
   \sum_{\substack{n=-\infty\\n\ne0}}^\infty
                                      \frac{(-1)^{n-1}q^{n(n+1)/2 + kn}}
                                           {(1-q^n)^{2k}}.
\mylabel{eq:symMid}
\end{equation}
\end{theorem}
\begin{proof}
As in the proof of  \cite[Theorem 2]{An07a} we have
\begin{align*}
\sum_{n=1}^\infty \mu_{2k}(n) q^n &
= \frac{1}{(2k)!} 
\left.\left(\left(\frac{\partial}{\partial z}\right)^{2k} z^{k-1} C(z,q) \right)\right\vert_{z=1}\\
&= \frac{1}{(2k)!}
   \sum_{j=0}^{k-1} \binom{2k}{j} (v-1) \cdots (v-j) C^{(2k-j)}(1,q)\\
&=\frac{1}{(q)_\infty}
   \sum_{j=0}^{k-1} \binom{k-1}{j} 
   \sum_{\substack{n=-\infty\\n\ne0}}^\infty
                                      \frac{(-1)^{n-1}q^{n(n-1)/2 + (2k-j)n}(1-q^n)}
                                           {(1-q^n)^{2k-j+1}},\\
&=\frac{1}{(q)_\infty}
   \sum_{\substack{n=-\infty\\n\ne0}}^\infty
                                      \frac{(-1)^{n-1}q^{n(n-1)/2 + 2kn}}
                                           {(1-q^n)^{2k}}
\left(1 + \frac{q^{-n}}{(1-q^n)^{-1}}\right)^{k-1}\\
&=\frac{1}{(q)_\infty} 
   \sum_{\substack{n=-\infty\\n\ne0}}^\infty
                                      \frac{(-1)^{n-1}q^{n(n+1)/2 + kn}}
                                           {(1-q^n)^{2k}}.
\end{align*}
\end{proof}

\section{Rank moments, crank moments and Bailey Chains}\label{sec:bailey}

In \cite{Pa10}, Alexander Patkowski used a limiting form
of Bailey's Lemma to obtain a partition identity analogous to 
\eqn{sptid}, which relates an spt-like function to the second rank
moment. We consider a similar limiting form that iterates Bailey's Lemma
and obtain a general theorem for Bailey pairs (see Theorem \thm{mainthm} below).
Then we show how our main identity \eqn{mainid} for
the difference between symmetrized crank and rank moments follows
from using well-known Bailey pairs. In this section we use the
standard notation found in \cite{Ga-Ra-book}.

\begin{definition}
\label{def:baileypair}
A pair of sequences $(\alpha_n(a,q),\beta_n(a,q))$ is called a Bailey pair
with parameters $(a,q)$ if
$$
\beta_n(a,q) = \sum_{r=0}^n \frac{\alpha_r(a,q)}{(q;q)_{n-r} (aq;q)_{n+r}}
$$
for all $n\ge 0$.
\end{definition}

\begin{theorem}[Bailey's Lemma]
Suppose  $(\alpha_n(a,q),\beta_n(a,q))$ is a Bailey pair
with parameters $(a,q)$.
Then $(\alpha_n'(a,q), \beta_n'(a,q))$
is another Bailey pair with
with parameters $(a,q)$, where
$$
\alpha_n'(a,q)=\frac{(\rho_1,\rho_2;q)_n}{(aq/\rho_1,aq/\rho_2;q)_n}
\left(\frac{aq}{\rho_1\rho_2}\right)^n \alpha_n(a,q)
$$
and
$$
\beta_n'(a,q)=
\sum_{k=0}^n \frac{(\rho_1,\rho_2;q)_k (aq/\rho_1\rho_2;q)_{n-k}}
{(aq/\rho_1,aq/\rho_2;q)_{n}(q;q)_{n-k}}
\left(\frac{aq}{\rho_1\rho_2}\right)^k\beta_k(a,q).
$$
\end{theorem}

For more information on Bailey's Lemma and its applications
see \cite[Ch.3]{An-CBMS-book}.
We will need the following limit which is an easy exercise.
\begin{equation}
\lim_{\rho_2\to1} 
\lim_{\rho_1\to1} 
\frac{1}{(1-\rho_1)(1-\rho_2)}
\left(1 - \frac{ (q)_k (q/\rho_1 \rho_2)_k}{(q/\rho_1)_k (q/\rho_2)_k}\right)
=
\sum_{j=1}^k \frac{q^j}{(1-q^j)^2}.
\mylabel{eq:easylimit}
\end{equation}

\begin{theorem}
\label{thm:mainthm}
Suppose  $(\alpha_n,\beta_n)=(\alpha_n(1,q),\beta_n(1,q))$ is a Bailey pair
with $a=1$,  and $\alpha_0=\beta_0=1$. Then
\begin{align*}
&\sum_{n_k \ge n_{k-1} \ge \cdots \ge n_1 \ge 1}
\frac{ (q)_{n_1}^2 q^{n_1 + n_2 + \cdots + n_k} \beta_{n_1}}
{(1-q^{n_k})^2 (1-q^{n_{k-1}})^2 \cdots (1-q^{n_1})^2}\\
&=
\sum_{n_k \ge n_{k-1} \ge \cdots \ge n_1 \ge 1}
\frac{q^{n_1 + n_2 + \cdots + n_k}}
{(1-q^{n_k})^2 (1-q^{n_{k-1}})^2 \cdots (1-q^{n_1})^2}
+ 
\sum_{r=1}^\infty \frac{q^{kr} \alpha_r}{ (1-q^r)^{2k}}.
\end{align*}
\end{theorem}
\begin{proof}
From Bailey's Lemma we have
\begin{align*}
\sum_{j=1}^n 
&\frac{ (\rho_1)_j (\rho_2)_j (q/\rho_1\rho_2)_{n-j} (q/\rho_1\rho_2)^j \beta_j}
     { (q)_{n-j}}\\
& = \frac{ (q/\rho_1)_n (q/\rho_2)_n}{(q)_n^2}
    \left(1 - \frac{ (q)_n (q/\rho_1\rho_2)_n }
                   { (q/\rho_1)_n (q/\rho_2)_n } \right) \\
& \qquad \qquad +
   (q/\rho_1)_n (q/\rho_2)_n 
   \sum_{r=1}^n \frac{ (\rho_1)_r (\rho_2)_r }
                {(q)_{n-r} (q)_{n+r} (q/\rho_1)_r (q/\rho_2)_r }
   \left( \frac{q}{\rho_1 \rho_2}\right)^r \alpha_r.
\end{align*}
We divide both sides by $(1-\rho_1)(1-\rho_2)$, let 
$\rho_1\to1$, $\rho_2\to1$, and use \eqn{easylimit} to obtain
$$                  
\sum_{j=1}^n (q)_{j-1}^2 q^j \beta_j 
 = \sum_{j=1}^n \frac{q^j}{(1-q^j)^2}
    + (q)_n^2 \sum_{r=1}^n \frac{ q^r \alpha_r}{(q)_{n-r} (q)_{n+r} (1-q^r)^2}.
$$
Letting $n\to\infty$ we have
$$                  
\sum_{j=1}^\infty (q)_{j-1}^2 q^j \beta_j 
 = \sum_{j=1}^\infty \frac{q^j}{(1-q^j)^2}
    + \sum_{r=1}^\infty \frac{ q^r \alpha_r}{(1-q^r)^2},
$$
which is the case $k=1$ of the theorem.

Now we suppose that the theorem is true for $k=K-1$, so that
\begin{align*}
&\sum_{n_K \ge n_{K-1} \ge \cdots \ge n_2 \ge 1}
\frac{ (q)_{n_2}^2 q^{n_2 + \cdots + n_K} \beta_{n_2}}
{(1-q^{n_K})^2 (1-q^{n_{K-1}})^2 \cdots (1-q^{n_2})^2}\\
&=
\sum_{n_K \ge n_{K-1} \ge \cdots \ge n_2 \ge 1}
\frac{q^{n_2 + \cdots + n_K}}
{(1-q^{n_K})^2 (1-q^{n_{K-1}})^2 \cdots (1-q^{n_2})^2}
+ 
\sum_{r=1}^\infty \frac{q^{(K-1)r} \alpha_r}{ (1-q^r)^{2K-2}}.
\end{align*}
We now replace $(\alpha_n,\beta_n)$ by the Bailey pair
$(\alpha_n',\beta_n')$ in Bailey's Lemma to obtain
\begin{align*}
&\sum_{\substack{n_K \ge n_{K-1} \ge \cdots \ge n_2 \ge 1\\ n_2\ge n_1\ge0}}
\frac{ (q)_{n_2}^2 q^{n_2 + \cdots + n_K} 
(\rho_1)_{n_1} (\rho_2)_{n_1} (q/\rho_1\rho_2)_{n_2-n_1} (q/\rho_1\rho_2)^{n_1} 
\beta_{n_1}} 
{(1-q^{n_K})^2 (1-q^{n_{K-1}})^2 \cdots (1-q^{n_2})^2 
(q)_{n_2-n_1} (q/\rho_1)_{n_2} (q/\rho_2)_{n_2}}\\
&=
\sum_{n_K \ge n_{K-1} \ge \cdots \ge n_2 \ge 1}
\frac{q^{n_2 + \cdots + n_K}}
{(1-q^{n_K})^2 (1-q^{n_{K-1}})^2 \cdots (1-q^{n_2})^2}
+ 
\sum_{r=1}^\infty \frac{q^{(K-1)r} (\rho_1)_r (\rho_2)_r (q/\rho_1\rho_2)^r \alpha_r}
{(1-q^r)^{2K-2} (q/\rho_1)_r (q/\rho_2)_r},
\end{align*}

and
\begin{align*}
&\sum_{n_K \ge n_{K-1} \ge \cdots \ge n_2 \ge n_1\ge1}
\frac{ (q)_{n_2}^2 q^{n_2 + \cdots + n_K} 
(\rho_1)_{n_1} (\rho_2)_{n_1} (q/\rho_1\rho_2)_{n_2-n_1} (q/\rho_1\rho_2)^{n_1} 
\beta_{n_1}} 
{(1-q^{n_K})^2 (1-q^{n_{K-1}})^2 \cdots (1-q^{n_2})^2 
(q)_{n_2-n_1} (q/\rho_1)_{n_2} (q/\rho_2)_{n_2}}\\
&=
\sum_{n_K \ge n_{K-1} \ge \cdots \ge n_2 \ge 1}
\frac{q^{n_2 + \cdots + n_K}}
{(1-q^{n_K})^2 (1-q^{n_{K-1}})^2 \cdots (1-q^{n_2})^2}
\left(1 - \frac{ (q)_{n_2} (q/\rho_1\rho_2)_{n_2}}
               { (q/\rho_1)_{n_2} (q/\rho_2)_{n_2} } \right) \\
&\qquad 
+ 
\sum_{r=1}^\infty \frac{q^{(K-1)r} (\rho_1)_r (\rho_2)_r (q/\rho_1\rho_2)^r \alpha_r}
{(1-q^r)^{2K-2} (q/\rho_1)_r (q/\rho_2)_r}.
\end{align*}
We divide both sides by $(1-\rho_1)(1-\rho_2)$, let
$\rho_1\to1$, $\rho_2\to1$, and use \eqn{easylimit} to obtain
\begin{align*}
&\sum_{n_K \ge n_{K-1} \ge \cdots \ge n_1 \ge 1}
\frac{ (q)_{n_1}^2 q^{n_1 + n_2 + \cdots + n_K} \beta_{n_1}}
{(1-q^{n_K})^2 (1-q^{n_{K-1}})^2 \cdots (1-q^{n_1})^2}\\
&=
\sum_{n_K \ge n_{K-1} \ge \cdots \ge n_1 \ge 1}
\frac{q^{n_1 + n_2 + \cdots + n_K}}
{(1-q^{n_K})^2 (1-q^{n_{K-1}})^2 \cdots (1-q^{n_1})^2}
+ 
\sum_{r=1}^\infty \frac{q^{Kr} \alpha_r}{ (1-q^r)^{2K}},
\end{align*}
which is the result for $k=K$. The general result follows by induction.
\end{proof}

\begin{cor}
\begin{equation}
\sum_{n_k \ge n_{k-1} \ge \cdots \ge n_1 \ge 1}
\frac{q^{n_1 + n_2 + \cdots + n_k}}
{(1-q^{n_k})^2 (1-q^{n_{k-1}})^2 \cdots (1-q^{n_1})^2}
=
\sum_{n=1}^\infty (-1)^{n-1} q^{n(n-1)/2 + kn} \frac{ (1+q^n) }
                                                    { (1-q^n)^{2k} }.
\mylabel{eq:symMid2}
\end{equation}
\end{cor}
\begin{proof}
The result follows from Theorem \thm{mainthm} using the well-known Bailey pair
\cite[pp.27-28]{An-CBMS-book}
$$
\alpha_n = \begin{cases} 1 & n=0,\\ (-1)^n q^{n(n-1)/2}(1+q^n) & n\ge1,
           \end{cases}\qquad
\beta_n = \begin{cases} 1 & n=0,\\ 0 & n\ge1. \end{cases}
$$
\end{proof}

We note that we can rewrite \eqn{symMid2} as
\beq
\sum_{n=1}^\infty \mu_{2k}(n) q^n = 
\frac{1}{(q)_\infty}
\sum_{n_k \ge n_{k-1} \ge \cdots \ge n_1 \ge 1}
\frac{q^{n_1 + n_2 + \cdots + n_k}}
{(1-q^{n_k})^2 (1-q^{n_{k-1}})^2 \cdots (1-q^{n_1})^2},
\mylabel{eq:Mukid}
\eeq
after using \eqn{symMid}.

\begin{cor}
\begin{align}
&\sum_{n_k \ge n_{k-1} \ge \cdots \ge n_1 \ge 1}
\frac{(q)_{n_1} q^{n_1 + n_2 + \cdots + n_k}}
{(1-q^{n_k})^2 (1-q^{n_{k-1}})^2 \cdots (1-q^{n_1})^2} 
\mylabel{eq:symNid2}\\
&=
\sum_{n_k \ge n_{k-1} \ge \cdots \ge n_1 \ge 1}
\frac{q^{n_1 + n_2 + \cdots + n_k}}
{(1-q^{n_k})^2 (1-q^{n_{k-1}})^2 \cdots (1-q^{n_1})^2}
+
\sum_{n=1}^\infty (-1)^{n} q^{n(3n-1)/2 + kn} \frac{ (1+q^n) }
                                                    { (1-q^n)^{2k} }.
\nonumber
\end{align}
\end{cor}
\begin{proof}
The result follows from Theorem \thm{mainthm} using the well-known Bailey pair
\cite[p.28]{An-CBMS-book}
$$
\alpha_n = \begin{cases} 1 & n=0,\\ (-1)^n q^{n(3n-1)/2}(1+q^n) & n\ge1, 
\end{cases}\qquad
\beta_n = \frac{1}{(q)_n}. 
$$
\end{proof}

\begin{cor}
\begin{align}
&\sum_{n=1}^\infty \left( \mu_{2k}(n) - \eta_{2k}(n)\right) q^n 
\mylabel{eq:sptkid}\\
&=
\sum_{n_k \ge n_{k-1} \ge \cdots \ge n_1 \ge 1}
\frac{q^{n_1 + n_2 + \cdots + n_k}}
{(1-q^{n_k})^2 (1-q^{n_{k-1}})^2 \cdots (1-q^{n_1})^2(q^{n_1+1};q)_\infty}.
\nonumber                
\end{align}
\end{cor}
\begin{proof}
After dividing both sides of \eqn{symNid2} by $(q)_\infty$ and using \eqn{symMid2} we 
have
\begin{align*}
&\sum_{n_k \ge n_{k-1} \ge \cdots \ge n_1 \ge 1}
\frac{q^{n_1 + n_2 + \cdots + n_k}}
{(1-q^{n_k})^2 (1-q^{n_{k-1}})^2 \cdots (1-q^{n_1})^2(q^{n_1+1};q)_\infty}
\\
&=
\frac{1}{(q)_\infty}\left(
\sum_{n=1}^\infty (-1)^{n-1} q^{n(n-1)/2 + kn} \frac{ (1+q^n) }
                                                    { (1-q^n)^{2k} }
-
\sum_{n=1}^\infty (-1)^{n-1} q^{n(3n-1)/2 + kn} \frac{ (1+q^n) }
                                                    { (1-q^n)^{2k} }
\right) \\
&=\sum_{n=1}^\infty \left( \mu_{2k}(n) - \eta_{2k}(n)\right) q^n,
\end{align*}
by \eqn{symMid} and \eqn{symNid}.
\end{proof}

\section{Rank and crank moment inequalities}\label{sec:MNineq}

In this section we prove the conjectured inequality \eqn{MNineq}
for rank and crank moments. We need to relate ordinary and
symmetrized moments. This is achieved by defining an analog
of Stirling numbers of the second kind. This approach was suggested by
Mike Hirschhorn.

We define a sequence of polynomials
$$
g_k(x) = \prod_{j=0}^{k-1} (x^2 - j^2),
$$
for $k\ge 1$. We want a sequence of numbers $S^{*}(n,k)$ such
that
$$
x^{2n} = \sum_{k=1}^n S^{*}(n,k) g_k(x),
$$
for $n\ge 1$. 
\begin{definition}
We define the sequence $S^{*}(n,k)$ ($1\ge k \ge n$) recursively by
\begin{enumerate}
\item
$S^{*}(1,1)=1$,
\item
$S^{*}(n,k)=0$ if $k\le 0$ or $k > n$, and
\item
$S^{*}(n+1,k) = S^{*}(n,k-1) + k^2 S^{*}(n,k)$,
for $1 \le k \le n+1$.
\end{enumerate}
\end{definition}

Below is a table of $S^{*}(n,k)$ for small n:

\begin{center}
$1$

$1$ \qquad $1$

$1$ \qquad $5$ \qquad $1$

$1$ \qquad $21$ \qquad $14$ \qquad $1$

$1$ \qquad $85$ \qquad $147$ \qquad $30$ \qquad $1$

$1$ \qquad $341$ \qquad $1408$ \qquad $649$ \qquad $55$ \qquad $1$
\end{center}

We note that if we replace $k^2$ by $k$ in the recurrence we obtain
the Stirling numbers of the second kind.
The numbers $S^{*}(n,k)$ first occur in a paper of 
MacMahon \cite[p.106]{MacM19}. Mikl\'os B\'ona reminded me that
Neil Sloane's Online Encylopedia of Integer Sequences \cite{Sl-OEIS}
can also handle $2$-dimensional sequences. One just needs to input the
first few terms of
\beq
\left\{\left\{ S^{*}(n,k)\right\}_{k=1}^n\right\}_{n=1}^\infty
= 1, 1, 1, 1, 5, 1, 1, 21, 14, 1, 1, 85, 147, 30, 1, \dots,
\mylabel{eq:Ssseq}
\eeq
to find the sequence labelled \texttt{A036969} \cite{Sl-A036969}, 
where more references
can be found.

We have
\begin{lemma}
\label{lem:stirling}
For $n\ge 1$,
$$
x^{2n} = \sum_{k=1}^n S^{*}(n,k) g_k(x).
$$
\end{lemma}
\begin{proof}
We proceed by induction on $n$. The result is true for
$n=1$ since $S^{*}(1,1)=1$ and $g_1(x)=x^2$. We now suppose the result
is true for $n=m$, so that
$$
x^{2m} = \sum_{k=1}^m S^{*}(m,k) g_k(x).
$$
We have $g_{k+1}(x) = (x^2 - k^2) g_k(x)$ and
$$
x^2 g_k(x) = g_{k+1}(x) + k^2 g_k(x),
$$
for $k\ge 1$. Thus     
\begin{align*}
x^{2m+2} &= \sum_{k=1}^m S^{*}(m,k) x^2 g_k(x) \\
         &= \sum_{k=1}^m S^{*}(m,k)( g_{k+1}(x) + k^2 g_k(x))\\
         &= \sum_{k=1}^{m+1} (S^{*}(m,k-1) + k^2 S^{*}(m,k)) g_k(x))\\
         &=  \sum_{k=1}^{m} S^{*}(m+1,k) g_k(x),
\end{align*}
and the result is true for $n=m+1$ and true for all $n$ by
induction.
\end{proof}

We can now express ordinary moments in terms of symmetrized moments.

\begin{theorem}
For $k\ge 1$              
\begin{align}
\mu_{2k}(n) &= \frac{1}{(2k)!} \sum_{m=-n}^n g_k(m) \,M(m,n),
\mylabel{eq:M2symM}\\
\eta_{2k}(n) &= \frac{1}{(2k)!} \sum_{m=-n}^n g_k(m) \,N(m,n),
\mylabel{eq:N2symN}\\
M_{2k}(n) &= \sum_{j=1}^k (2j)! \, S^{*}(k,j) \, \mu_{2j}(n),
\mylabel{eq:symM2M}\\
N_{2k}(n) &= \sum_{j=1}^k (2j)! \, S^{*}(k,j) \, \eta_{2j}(n).  
\mylabel{eq:symN2N}
\end{align}
\end{theorem}
\begin{proof}
Suppose $k\ge1$. Then
\begin{align*}
\mu_{2k}(n) &= \sum_{m=-n}^n \binom{m+k-1}{2k} M(m,n)\\
&= \frac{1}{(2k)!}\sum_{m=-n}^n (m+k-1)(m+k-2)\cdots(m-k) M(m,n)\\
&= \frac{1}{(2k)!}\sum_{m=-n}^n 
(m^2-(k-1)^2)(m^2-(k-2)^2)\cdots(m^2-1)m(m-k) M(m,n)\\
&= \frac{1}{(2k)!}\sum_{m=-n}^n g_k(m) M(m,n), 
\end{align*}
since $M(-m,n)=M(m,n)$ for all $m$. This gives \eqn{M2symM} and
similarly \eqn{N2symN}.
Using Lemma \lem{stirling} and \eqn{M2symM} we see that
\begin{align*}
M_{2k} &= \sum_{m=-n}^n m^{2k} M(m,n)\\
       &= \sum_{m=-n}^n \left(\sum_{j=1}^k S^{*}(k,j) g_j(m)\right)M(m,n)\\
       &= \sum_{j=1}^k (2j)! \, S^{*}(k,j) \, \mu_{2j}(n),
\end{align*}
which is \eqn{symM2M}. Equation \eqn{symN2N} follows similarly.
\end{proof}

We can now deduce our crank-rank moment inequality.

\begin{cor}
For $1\le k \le n$ 
$$
M_{2k}(n) > N_{2k}(n).
$$
\end{cor}
\begin{proof}
Suppose $k\ge1$. Then from \eqn{sptkid} we have
$$
\sum_{n=1}^\infty \left( \mu_{2j}(n) - \eta_{2j}(n)\right) q^n  =
\frac{q^j}{(1-q)^{2j} (q^2;q)_\infty} + \cdots,
$$
and we see that
$$
\mu_{2j}(n) > \eta_{2j}(n),
$$
for all $n\ge j\ge 1$. Now using \eqn{symM2M}, \eqn{symN2N} and the 
fact that the coefficients $S^{*}(k,j)$ are positive integers we have
\begin{align*}
M_{2k}(n) - N_{2k}(n) &=
\sum_{j=1}^k (2j)!\, S^{*}(k,j) \,\left(\mu_{2j}(n) - \eta_{2j}(n)\right)\\
& \ge 2 \left(\mu_{2}(n) - \eta_{2}(n)\right) > 0,
\end{align*}
for all $n\ge 1$.
\end{proof}

\section{Higher order spt-functions}\label{sec:sptk}

In this section we define
a higher-order spt function $\spt_k(n)$ so that
$$
\spt_k(n) = \mu_{2k}(n) - \eta_{2k}(n),
$$
for all $k\ge1$ and $n\ge1$. 
The idea is to interpret the right side of \eqn{sptkid} in terms of 
partitions. 

\begin{definition}
For a partition $\pi$ with $m$ different  parts
$$
n_1 < n_2 < \cdots < n_m,
$$
we define $f_j=f_j(\pi)$ to be the frequency of part $n_j$ for 
$1 \le j \le m$.
\end{definition}
We note that $f_1=f_1(\pi)$ is the number of smallest parts in the partition
$\pi$ and Andrews' function
$$
\spt(n) = \sum_{\pi \vdash n} f_1(\pi).
$$

\begin{definition}
Let $k\ge 1$. For a partition $\pi$ we define a weight
$$
\omega_k(\pi) = 
\sum_{\substack{ m_1 + m_2 + \cdots + m_r = k \\ 1 \le r \le k}}
\binom{f_1 + m_1 -1}{2m_1 -1}
\sum_{2 \le j_2 < j_3 < \cdots < j_r}
\binom{f_{j_2} + m_2}{2m_2} \binom{f_{j_3} + m_3}{2m_3}
\cdots
\binom{f_{j_r} + m_r}{2m_r},
$$
and
$$
\spt_k(n) = \sum_{\pi \vdash n} \omega_k(\pi).
$$
\end{definition}
We note that the outer sum above is over all compositions 
$m_1 + m_2 + \cdots + m_r$ of $k$.

\begin{example}[$k=1$]
There is only one composition of $1$, $\omega_1(\pi)=f_1(\pi)$ and
$$
\spt_1(n) = \spt(n).
$$
\end{example}

\begin{example}[$k=2$]
There are two compositions of $2$, namely $2$ and $1+1$, 
$$
\omega_2(\pi) = \binom{f_1+1}{3} + f_1 \sum_{2\le j} \binom{f_j+1}{2},
$$
and
$$
\spt_2(n) = \sum_{\pi \vdash n} \omega_2(\pi).
$$
We calculate $\spt_2(4)$. There are five partitions of $4$:
$$
\begin{array}{lll}
4 & f_1=1 & \omega_2 = 0\\
3+1 & f_1=f_2=1 & \omega_2 = 1\\
2+2 & f_1=2 & \omega_2 = 1\\
2+1+1 & f_1=2, f_2=1 & \omega_2 = 1+2=3\\
1+1+1+1 & f_1=4 & \omega_2 = 10
\end{array}
$$
Hence $\spt_2(4) = 0 + 1 + 1 + 3 + 10 = 15$.
\end{example}

\begin{example}[$k=3$]
There are four compositions of $3$, namely $3$, $2+1$,  $1+2$ and $1+1+1+1$.
Hence the definition of $\omega_3(\pi)$ has four terms:
$$
\omega_3(\pi) = \binom{f_1+2}{5} + 
                 \binom{f_1 + 1}{3} \sum_{2\le j} \binom{f_j+1}{2} +
                 f_1 \sum_{2\le j} \binom{f_j+2}{4} +
                 f_1 \sum_{2\le j < k} \binom{f_j+1}{2}\binom{f_k+1}{2},
$$
and
$$
\spt_3(n) = \sum_{\pi \vdash n} \omega_3(\pi).
$$
To illustrate, we calculate $\spt_3(5)$. There are seven partitions of $5$:
$$
\begin{array}{lll}
5 & f_1=1 & \omega_3 = 0\\
4+1 & f_1=f_2=1 & \omega_3 = 0\\
3+2 & f_1=f_2=1 & \omega_3 = 0\\
3+1+1 & f_1=2, f_2=1 & \omega_3 = 1\\
2+2+1 & f_1=1, f_2=2 & \omega_3 = 1\\
2+1+1+1 & f_1=3, f_2=1 & \omega_3 = 1 + 4 = 5\\
1+1+1+1+1 & f_1=5& \omega_3 = 21\\
\end{array}
$$
Hence $\spt_3(5) = 0 + 0 + 0 + 1 +1 + 5 + 21 = 28$.
\end{example}

Our goal in this section is to prove
\begin{theorem}
\label{thm:sptkthm}
For $1\le k \le n$
$$
\spt_k(n) = \mu_k(n) - \eta_k(n).
$$
\end{theorem}
\begin{proof}
First we need the elementary
identities
$$                   
\sum_{n=j}^\infty \binom{n+j-1}{2j-1}x^n = \frac{x^j}{(1-x)^{2j}}\quad
\mbox{and}\quad 
\sum_{n=j}^\infty \binom{n+j}{2j}x^n = \frac{x^j}{(1-x)^{2j+1}}.
$$
To give the idea of the proof we first consider the 
case $k=4$. 
From \eqn{sptkid} we have
\begin{align*}
&\sum_{n=4}^\infty \left( \mu_{8}(n) - \eta_{8k}(n)\right) q^n \\
& = 
\sum_{1\le  m \le j \le k \le n}
 \frac{q^{m+j+k+n}}{
(1-q^m)^2 (1-q^j)^2 (1-q^k)^2 (1-q^n)^2 (q^{m+1};q)_\infty}\\
& = 
\sum_{1 \le m=j=k=n } + 
\sum_{1 \le m=j=k<n } + 
\sum_{1 \le m=j<k=n } + 
\sum_{1 \le m<j=k=n } + \\ 
& \quad \sum_{1 \le m=j<k<n } + 
\sum_{1 \le m<j=k<n } + 
\sum_{1 \le m<j<k=n } + 
\sum_{1 \le m<j<k<n }   
\\
& =
\sum_{m=1}^\infty \frac{q^{4m}}{(1-q^m)^8}
\prod_{i> m} \frac{1}{(1-q^i)}
+
\sum_{1\le m < n} \frac{q^{3m}}{(1-q^m)^6} \frac{q^n}{(1-q^n)^3} 
\prod_{\substack{i> m\\ i\ne n}} \frac{1}{(1-q^i)} \\
& +
\sum_{1\le m < n} \frac{q^{2m}}{(1-q^m)^4} \frac{q^{2n}}{(1-q^n)^5} 
\prod_{\substack{i> m\\ i\ne n}} \frac{1}{(1-q^i)} 
 +
\sum_{1\le m < n} \frac{q^{m}}{(1-q^m)^2} \frac{q^{3n}}{(1-q^n)^7} 
\prod_{\substack{i> m\\ i\ne n}} \frac{1}{(1-q^i)} \\
& +
\sum_{1\le m < k < n} \frac{q^{2m}}{(1-q^m)^4} \frac{q^{k}}{(1-q^k)^3} \frac{q^{n}}{(1-q^n)^3} 
\prod_{\substack{i> m\\ i\ne k,n}} \frac{1}{(1-q^i)} 
\\
& +
\sum_{1\le m < k < n} \frac{q^{m}}{(1-q^m)^2} \frac{q^{2k}}{(1-q^k)^5} \frac{q^{n}}{(1-q^n)^3} 
\prod_{\substack{i> m\\ i\ne k,n}} \frac{1}{(1-q^i)} \\
& +
\sum_{1\le m < k < n} \frac{q^{m}}{(1-q^m)^2} \frac{q^{k}}{(1-q^k)^3} \frac{q^{2n}}{(1-q^n)^5} 
\prod_{\substack{i> m\\ i\ne k,n}} \frac{1}{(1-q^i)} \\
& +
\sum_{1\le m < j < k < n} \frac{q^{m}}{(1-q^m)^2} \frac{q^{j}}{(1-q^k)^3} 
\frac{q^{k}}{(1-q^k)^3} \frac{q^{n}}{(1-q^n)^3} 
\prod_{\substack{i> m\\ i\ne j,k,n}} \frac{1}{(1-q^i)}.
\end{align*}
There are eight compositions of $4$: $4$, $3+1$, $2+2$, $1+3$, $2+1+1$,
$1+2+1$, $1+1+2$, and $1+1+1+1$.
Each of the eight sums above has the form 
$$
\sum_{1 \le n_1 < n_{j_2} < \cdots < n_{j_r}}
\frac{q^{m_1 n_1}}{(1-q^{n_1})^{2 m_1}}
\frac{q^{m_2 n_{j_2}}}{(1-q^{n_{j_2}})^{2 m_2+1}}
\cdots
\frac{q^{m_r n_{j_r}}}{(1-q^{n_{j_r}})^{2 m_r+1}}
\prod_{\substack{i> n_1\\ i\not\in\{n_{j_2},\dots,n_{j_r}\}}} \frac{1}{(1-q^i)},
$$
where $m_1 + m_2 + \cdots + m_r$ is a composition of $k=4$.
This sum can be written as
\begin{align*}
&\sum_{\substack{{1 \le n_1 < n_{j_2} < \cdots < n_{j_r}}\\
{f_1\ge m_1, f_{j_2}\ge m_{2}, \dots f_{j_r}\ge m_{r}} }}
\binom{f_1 + m_1 -1 }{2m_1 - 1} 
\binom{f_{j_2} + m_2 -1 }{2m_2} 
\cdots 
\binom{f_{j_r} + m_r -1 }{2m_r} \\
&\qquad \qquad \qquad \qquad \qquad \qquad\cdot \, q^{f_1 n_1 + f_{j_2} n_{j_2} + \cdots + f_{j_r} n_{j_r}}
\prod_{\substack{i> n_1\\ i\not\in\{n_{j_2},\dots,n_{j_r}\}}} \frac{1}{(1-q^i)}.
\end{align*}
We see that this is the generating function for certain weighted partitions
in which $n_1$ is the smallest part, $n_1 < n_{j_2} < \cdots < n_{j_r}$ is an
$r$-subset of the parts of the partition, and $f_j$ is the frequency of part $n_j$
for each $j$. It follows that
$$                    
\sum_{n=1}^\infty \left( \mu_{8}(n) - \eta_{8k}(n)\right) q^n 
= 
\sum_{\pi \vdash n} \omega_4(\pi) = \sum_{n=4}^\infty \spt_4(n) q^n.
$$ 
The proof of the general case is completely analogous.
Now suppose $k\ge 1$. From \eqn{sptkid} we have
\begin{align*}
&\sum_{n=1}^\infty \left( \mu_{2k}(n) - \eta_{2k}(n)\right) q^n \nonumber\\
&=
\sum_{1\le n_1 \le n_{2} \le \cdots \le n_{k}}
\frac{q^{n_1 + n_2 + \cdots + n_k}}
{(1-q^{n_1})^2 (1-q^{n_{2}})^2 \cdots (1-q^{n_k})^2(q^{n_1+1};q)_\infty}.
\end{align*}
We partition this sum into $2^{k-1}$ subsums by changing each ``$\le$" in 
the general inequality 
${n_1 \le n_{2} \le \cdots \le n_k}$ to either ``$=$" or ``$<$".
In this way each subsum corresponds to a unique composition 
$m_1 + m_2 + \cdots + m_r$ of $k$ (where $1 \le r \le k$).
We proceed just as in the case $k=4$ and the general result follows.
\end{proof}

\section{Congruences for higher order spt-functions}\label{sec:congsptk}

In \cite{Br-Ga-Ma} it was shown that given any prime $\ell>3$ with $k$ and $j$ fixed 
there are infinitely many arithmetic progressions $An+B$ such that
$$
\eta_{2k}(An + B) \equiv 0 \pmod{\ell^j}.
$$
Using known results for crank moments \cite[\S 7]{Br-Ga-Ma} and standard
techniques \cite{Br-Ga-Ma}, \cite{Br08} we may deduce the analog of this result for 
higher order spt-functions.
In this section we prove a number of nice explicit congruences for  higher order 
spt-functions.
Many of the congruences follow from known results for rank and
crank moments \cite{At-Ga}.      

\begin{theorem}
\begin{align}
\spt_2(n) &\equiv 0 \pmod{5},\quad \mbox{if $n\equiv 0,1,4\pmod{5}$}
\mylabel{eq:spt2mod5}\\
\spt_2(n) &\equiv 0 \pmod{7},\quad \mbox{if $n\equiv 0,1,5\pmod{7}$}
\mylabel{eq:spt2mod7}\\
\spt_2(n) &\equiv 0 \pmod{11},\quad \mbox{if $n\equiv 0\pmod{11}$}.
\mylabel{eq:spt2mod11}
\end{align}
\end{theorem}
\begin{proof}
By definition,
$$
\spt_2(n) = \mu_4(n) - \eta_4(n) = \frac{1}{24}(M_4(n) - M_2(n) - N_4(n) 
+ N_2(n)).
$$
From \cite[(5.6)]{At-Ga} we have
$$
N_4(n) = -\frac{2}{3}(3n+1) M_2(n) + \frac{8}{3} M_4(n) + (1-12n) N_2(n),
$$
and
\begin{equation}
24 \spt_2(n) = (2n - \tfrac{1}{3}) M_2(n) - \frac{5}{3} M_4(n) + 12n N_2(n).
\mylabel{eq:spt2id}
\end{equation}
The congruence \eqn{spt2mod5} now follows from
\begin{align*}
M_2(n) &= 2n p(n), \qquad \mbox{(\cite[(1.27)]{At-Ga}}\\
N_2(n) &\equiv (n+4) p(n),\quad \mbox{for $n\not\equiv0,3\pmod{5}$}
\qquad \mbox{(\cite[p.285]{Ga10a})},\\
p(5n+4) &\equiv 0 \pmod{5}.
\end{align*}

To begin the proof of \eqn{spt2mod7} we use \eqn{spt2id} to obtain
$$
\spt_2(n) \equiv M_4(n) + 3(n+1) M_2(n) + 4n N_2(n) \pmod{7}.
$$
From \cite[p.285]{Ga10a}
\begin{equation}
N_2(n) \equiv (6n+1) p(n) \pmod{7},\quad \mbox{for $n\not\equiv0,2,6\pmod{7}$}
\mylabel{eq:N2mod7}
\end{equation}
so that
\begin{equation}
\spt_2(n) \equiv M_4(n) + 3(n+1) M_2(n)\pmod{7},\quad 
\mbox{for $n\equiv0,1,5\pmod{7}$.}
\mylabel{eq:spt2mod7id}
\end{equation}
From \cite[(1.21)]{At-Ga} we have
$$
M_4(7n+5) \equiv M_2(7n+5) \equiv 0\pmod{7},\quad\mbox{and}\quad
\spt_2(7n+5)\equiv0\pmod{7}.
$$
From \cite[(6.5)]{At-Ga}
\begin{equation}
(n+2) M_4(n) \equiv -(6n^2+4n+1) M_2(n)\pmod{7},
\mylabel{eq:M42mod7}
\end{equation}
so that
\begin{align}
M_4(7n)&\equiv 3 M_2(7n) \equiv 0 \pmod{7},\quad \mbox{(since $M_2(n)=2np(n)$)},
\mylabel{eq:M420mod7}\\
M_4(7n+1) &\equiv M_2(7n+1) \pmod{7},
\mylabel{eq:M421mod7}
\end{align}
and
$$
\spt_2(7n)\equiv\spt_2(7n+1)\equiv0\pmod{7},
$$
by \eqn{spt2mod7id}.

The proof of \eqn{spt2mod11} is similar to that of \eqn{spt2mod5}
and \eqn{spt2mod7}.
From \eqn{spt2id} we have
$$
\spt_2(n) \equiv M_4(n) + (n+9) M_2(n) + 6n N_2(n) \pmod{11}.
$$
From \cite[(6.6)]{At-Ga}
$$
(n+5)^3 M_4(n) \equiv (5n^4 + 10n^3 + 8n^2 + 8n+ 9) M_2(n)\pmod{11},
$$
so that
$$
M_4(11n)\equiv M_2(11n)\equiv0\pmod{11},
$$
and
$$
\spt_2(11n)\equiv0\pmod{11}.
$$
\end{proof}

\begin{theorem}
\begin{align}
\spt_3(n) &\equiv 0 \pmod{7},\quad \mbox{if $n\not\equiv 3,6\pmod{7}$},
\mylabel{eq:spt3mod7}\\
\spt_3(n) &\equiv 0 \pmod{2},\quad \mbox{if $n\equiv 1\pmod{4}$.}
\mylabel{eq:spt3mod2}
\end{align}
\end{theorem}
\begin{proof}
From \cite[(5.6)-(5.7)]{At-Ga} and the definition of $\spt_3(n)$ we have
\begin{align}            
\spt_3(n) & = - \frac{7}{7920} M_6(n) 
            + \frac{1}{1584}(60n+13) M_4(n)
            + \frac{1}{3960}(7-78n-108n^2) M_2(n) \nonumber\\
          &\qquad  - \frac{1}{20} n(1+3n) N_2(n),
\mylabel{eq:spt3id}
\end{align}
and
\begin{equation}
\spt_3(n) \equiv n(5n+4) M_2(n) + (3 + 2n) M_4(n) + n(3n+1) N_2(n)\pmod{7}.
\mylabel{eq:spt3mod7id}
\end{equation}
This implies that
$$
\spt_3(7n+2)\equiv0\pmod{7}.
$$
Known results for the rank and crank \cite[(1.18),(1.21)]{At-Ga} imply that 
$$
\spt_3(7n+5)\equiv0\pmod{7}.
$$
The congruences \eqn{N2mod7}, \eqn{M420mod7}, \eqn{M421mod7} and 
\eqn{spt3mod7id} imply that
$$
\spt_3(7n)\equiv\spt_3(7n+1)\equiv0\pmod{7}.
$$
The congruences \eqn{M42mod7} and \eqn{spt3mod7id} imply that
$$
\spt_3(7n+4)\equiv2M_2(7n+4) + 3N_2(7n+4)\pmod{7}.
$$
From \eqn{N2mod7} and the fact that $M_2(n) = 2np(n)$ we have
$$
M_2(7n+4)\equiv p(7n+4), \quad N_2(7n+4) \equiv 4 p(7n+4) \pmod{7}
$$
and
$$
\spt_3(7n+4)\equiv0 \pmod{7}.
$$

   We now turn to the congruence \eqn{spt3mod2}. First we note that the 
term 
$$
 \frac{1}{20} n(1+3n) N_2(n) \equiv 0 \pmod{2},
$$
when $n\equiv 1 \pmod{4}$ since $N_2(n)\equiv0\pmod{2}$.
We define
$$
s_3(n) =  - \frac{7}{7920} M_6(n) 
            + \frac{1}{1584}(60n+13) M_4(n)
            + \frac{1}{3960}(7-78n-108n^2) M_2(n) 
$$
so that
$$
\spt_3(4n+1) \equiv s_3(4n+1) \pmod{2}.
$$
By \cite[Theorem 4.2]{At-Ga}, the function
$$
S_3(q) := \sum_{n=1}^\infty s_3(n) q^n \in P \mathcal{W}_3,
$$
where $\mathcal{W}_n$ is a space  of quasimodular forms
of weight bounded by $2n$ defined in \cite[(3.7)]{At-Ga}, and
\beq
P = P(q) = \frac{1}{(q)_\infty}.
\mylabel{eq:Pdef}
\eeq
We define the functions
$$
P_3(q) = \sum_{n=1}^\infty p_3(n) q^n := P(q) \sum_{n=1}^\infty \sigma_3(n) q^n,
$$
and
$$
P_5(q) = \sum_{n=1}^\infty p_5(n) q^n := P(q) \sum_{n=1}^\infty \sigma_5(n) q^n.
$$
Let $\delta_q = q\frac{d}{dq}$. 
By \cite[(3.29) and Lemma 4.1]{At-Ga} the functions
$\delta_q(P)$, $\delta_q^2(P)$, $\delta_q^3(P)$,
$P_3$, $\delta_q(P_3)$, and $P_5 \in P \mathcal{W}_3$. Since
$\dim \mathcal{W}_3 = 6$ by \cite[Cor.3.6]{At-Ga}, there is a linear
relation between these functions and $S_3(q)$. A calculation gives
that
$$
s_3(n) = \frac{n}{270}(5 - 12n - 147n^2) p(n) + \frac{1}{12}(6n+1) p_3(n)
 - \frac{7}{540} p_5(n)
$$
and
$$
4 s_3(n) \equiv 6n (1+n^2) p(n) + (3+2n) p_3(n) + 7 p_5(n) \pmod{8}.
$$
Since $d^3\equiv d^5\pmod{8}$ it follows that
$$
\sigma_3(n)\equiv\sigma_5(n) \pmod{8}\quad\mbox{and}\quad
p_3(n)\equiv p_5(n) \pmod{8}.                          
$$
Hence
$$
4 s_3(n) \equiv 6n (1+n^2) p(n) + (10+2n) p_3(n)\pmod{8},
$$ 
and
$$
s_3(4n + 1) \equiv p(4n+1) + p_3(4n+1)\pmod{2}.
$$ 
It is well known that
$$
\delta_q(P) = \sum_{n=1}^\infty n p(n) q^n = P(q) \sum_{n=1}^\infty
\sigma(n) q^n.
$$
Since $\sigma(n)\equiv\sigma_3(n)\pmod{2}$ it follows that
$$
n p(n) \equiv p_3(n) \pmod{2},
$$
$$
p(4n+1) \equiv p_3(4n+1) \pmod{2},
$$
and
$$
s_3(4n+1)\equiv 0 \pmod{2},
$$
which completes the proof of \eqn{spt3mod2}.
\end{proof}

\begin{theorem}
\label{thm:spt4mod3}
\begin{equation}
\spt_4(3n) \equiv 0 \pmod{3}.
\mylabel{eq:spt4mod3}
\end{equation}
\end{theorem}
\begin{proof}
From \eqn{symNid} and \eqn{symMid} we have
\begin{align*}
\sum_{n=1}^\infty \spt_4(n) q^n
& = \frac{1}{(q)_\infty}
    \left(
   \sum_{\substack{n=-\infty\\n\ne0}}^\infty
                                      \frac{(-1)^{n-1}q^{n(n+1)/2 + 4n}}
                                           {(1-q^n)^{8}}
   -\sum_{\substack{n=-\infty\\n\ne0}}^\infty
                                      \frac{(-1)^{n-1}q^{n(3n+1)/2 + 4n}}
                                           {(1-q^n)^{8}}
    \right)\\
&\equiv
 \frac{1}{(q)_\infty}
    \left(
   \sum_{\substack{n=-\infty\\n\ne0}}^\infty
                \frac{(-1)^{n-1}q^{n(n+1)/2 + 4n}(1-q^n)}
                                           {(1-q^{9n})}  
    \right.
   \\
& \qquad\qquad
  \left.
   -\sum_{\substack{n=-\infty\\n\ne0}}^\infty
                 \frac{(-1)^{n-1}q^{n(3n+1)/2 + 4n}(1-q^n)}
                                         {(1-q^{9n})}
    \right)
   \pmod{3}.
\end{align*}
Before we can proceed we need some results for the rank and
crank mod $9$. We define
$$
S_k(b) = S_k(b,t) := \
   \sum_{\substack{n=-\infty\\n\ne0}}^\infty
                 \frac{(-1)^{n}q^{n(kn+1)/2 + bn}}
                                         {(1-q^{tn})},
$$
so that
\begin{align*}
\sum_{n=1}^\infty \spt_4(n) q^n
&\equiv \frac{1}{(q)_\infty}\left(-S_1(4,9) + S_1(5,9) + S_3(4,9) - S_3(5,9)
\right)
\pmod{9}.
\end{align*}
Now let $M(r,t,n)$ denote the number of partitions of $n$ with crank
congruent to $r$ mod $t$ and
let $N(r,t,n)$ denote the number of partitions of $n$ with rank
congruent to $r$ mod $t$. Then by \cite[(2.13)]{At-Sw} and 
\cite[(2.5)]{Ek00}
we have
$$
\sum_{n=0}^\infty N(r,t,n)q^n = \frac{1}{(q)_\infty}\left(S_3(r,t)+S_3(t-r,r)
\right)
$$
and
$$
\sum_{n=0}^\infty M(r,t,n)q^n = \frac{1}{(q)_\infty}\left(S_1(r,t)+S_1(t-r,r)
\right).
$$
From \cite[(2.3)]{Ek00} and \cite[(6.2)]{At-Sw} 
$$
S_k(b,t) = - S_k(t-1-b),
$$
for $k=1,3$.
Hence
$$
\sum_{n=0}^\infty M(4,9,n)q^n = \frac{1}{(q)_\infty}(S_1(4,9) + S_1(5,9))
=
\frac{1}{(q)_\infty} S_1(5,9)
$$
and
$$
\sum_{n=0}^\infty N(4,9,n)q^n = \frac{1}{(q)_\infty}(S_3(4,9) + S_3(5,9))
=
\frac{1}{(q)_\infty} S_3(5,9)
$$
since
$$
S_1(4,9) = S_3(4,9) = 0.
$$
It follows that
$$
\spt_4(n) \equiv M(4,9,n) - N(4,9,n) \pmod{3}. 
$$
Lewis \cite[(1a)]{Le92} has shown that
$$
M(4,9,3n) = N(4,9,3n)
$$
and our congruence \eqn{spt4mod3} follows.
\end{proof}

If we try the approach of using quasimodular forms to the prove
the congruence \eqn{spt4mod3} we are led to a congruence
for the Ramanujan tau-function.
\begin{cor}
\begin{align}
\qquad\qquad \tau(n) &\equiv
 \left( 588+297\,n+258\,{n}^{2} +9\,{n}^{3} +108\,{n}^{4}+486\,{n}^{5}\right) 
 \sigma_{{1}} \left( n \right) \nonumber\\
& + \left( 60+255\,n+189\,{n}^{2}+612\,{n}^{3}+162\,{n}^{4} \right) 
    \sigma_{{3}} \left( n\right) 
\mylabel{eq:taucong}\\
& + \left(306+297\,n+ 540\,{n}^{2}+180\,{n}^{3} \right) 
     \sigma_{{5}} \left( n \right) 
 + \left( 177+576\,n+454\,{n}^{2}\right) \sigma_{{7}} \left( n \right)\nonumber\\
&   + \left(201+ 690\,n\right) \sigma_{{9}} \left( n \right) 
+117\,\sigma_{{11}}
 \left( n \right) 
\pmod{3^6}.
\nonumber
\end{align}
\end{cor}
\begin{proof}
From \cite[(5.6)-(5.8)]{At-Ga} and the definition of $\spt_3(n)$ we see that
\begin{align}            
\spt_4(n) &=
-{\frac {67}{7362432}}M_8(n)
+{\frac{1}{2629440}}(491+1176n)M_6(n) \nonumber\\
&-{\frac{1}{1051776}}(1309+8400n+5856{n}^{2})M_4(n)\nonumber\\
&+{\frac{1}{3067680}}(-851+10966n+21204{n}^{2}+12162{n}^{3})M_2(n)\nonumber\\
&+{\frac{1}{140}}(n+4{n}^{2}+3{n}^{3})N_2(n).
\mylabel{eq:spt4id}
\end{align}
We define
\begin{align*}            
s_4(n) &=
-{\frac {67}{7362432}}M_8(n)
+{\frac{1}{2629440}}(491+1176n)M_6(n) \nonumber\\
&-{\frac{1}{1051776}}(1309+8400n+5856{n}^{2})M_4(n)\nonumber\\
&+{\frac{1}{3067680}}(-851+10966n+21204{n}^{2}+12162{n}^{3})M_2(n)\nonumber
\end{align*}
so that
$$
\spt_4(3n)\equiv s_4(3n) \pmod{3}.
$$
By \cite[Theorem 4.2]{At-Ga}, the function
$$
S_4(q) := \sum_{n=1}^\infty s_4(n) q^n \in P \mathcal{W}_4,
$$
$$
S^{*}_4(q) := \left(\delta_q^2-1\right)S_4(q) = 
\sum_{n=1}^\infty (n^2-1) s_4(n) q^n \in P \mathcal{W}_6,
$$
and
\begin{equation}
S^{*}_4(q) \equiv 0\pmod{3},
\mylabel{eq:S4starmod3}
\end{equation}
by Theorem \thm{spt4mod3}.
By \cite[(3.29)]{At-Ga} the functions
$\delta_q^j(\Phi_{2k+1})$ ($0 \le j \le 5-k$, $0 \le k \le 5$), 
and  $\Delta \in \mathcal{W}_6$, 
where 
$$
\Phi_j=\Phi_j(q) = 
\sum_{n=1}^\infty \frac{n^j q^n}{1-q^n} = \sum_{m,n\ge1} n^j q^{nm}
= \sum_{n=1}^\infty \sigma_j(n) q^n,
$$
and
$$
\Delta = \Delta(q) = \sum_{n=1}^\infty \tau(n) q^n = q\prod_{n=1}^\infty (1-q^n)^{24}.
$$
Since
$\dim \mathcal{W}_6 = 22$ by \cite[Cor.3.6]{At-Ga}, there is a linear
relation between these functions and $S^{*}_4(q)/P$. 
In fact, we can write the function $S^{*}_4(q)/P$ as a linear combination of the
$22$ functions $\delta_q^j(\Phi_{2k+1})$ ($0 \le j \le 5-k$, $0 \le k \le 5$),
and  $\Delta \in \mathcal{W}_6$. The coefficients in this linear combination are
rational numbers, and we find that we need to multiply each coefficient by $3^5$
to obtain $3$-integral rationals. The congruence \eqn{S4starmod3} then implies
a congruence mod $3^6$ between the arithmetic functions 
$n^j(\sigma_{2k+1}(n))$ ($0 \le j \le 5-k$, $0 \le k \le 5$),
and  $\tau(n)$.   Solving this congruence for $\tau(n)$ gives the result
\eqn{taucong}.                
\end{proof}

Ashworth \cite{As-thesis} (see also \cite{Ko76}) has also obtained
congruences for $\tau(n)$ mod powers of $3$. Ashworth's congruences
have a different form and depend on the residue of $n$ mod $3$.


\section{Concluding remarks}\label{sec:remarks}

It should be pointed out that Bringmann, Mahlburg and Rhoades \cite{Br-Ma-Rh}
have proved that there are positive constants 
$\alpha_{k}$ and $\beta_{k}$ such 
that
\begin{align}
M_{2k}(n) \sim N_{2k}(n) &\sim \alpha_{k} n^k \, p(n) 
\mylabel{eq:asymp1}\\
M_{2k}(n) - N_{2k}(n) &\sim \beta_{k} n^{k-\frac{1}{2}} \, p(n),  
\mylabel{eq:asymp2}
\end{align}
as $n\to\infty$ when $k$ is fixed.
This implies that

\eject

\begin{equation}
\spt_k(n) \sim \frac{\beta_{k}}{(2k)!}  n^{k-\frac{1}{2}} \, p(n),
\mylabel{eq:sptkapprox}
\end{equation}
as $n\to\infty$ when $k$ is fixed. It would interesting to consider
whether the new identity \eqn{mainid} could lead to
an elementary upper bound for $\spt_k(n)$.

Folsom and Ono \cite{Fo-On} found nontrivial congruences
for Andrews spt-function mod $2$ and $3$. Ono \cite{On10}
also found simple explicit congruences for Andrews' spt-function
modulo every prime $>3$. These congruences are related to
the action of a  weight $\tfrac{3}{2}$ Hecke operator. It would
be interesting to determine whether such behavior continues
for the higher degree spt-functions and higher weight
Hecke operators.

The function
\beq
A_k(q) =
\sum_{n_k \ge n_{k-1} \ge \cdots \ge n_1 \ge 1}
\frac{q^{n_1 + n_2 + \cdots + n_k}}
{(1-q^{n_k})^2 (1-q^{n_{k-1}})^2 \cdots (1-q^{n_1})^2}
\mylabel{eq:Akdef}           
\eeq
occurs in equation \eqn{Mukidintro} so that
\beq
\sum_{n=1}^\infty \mu_{2k}(n) q^n =
\frac{1}{(q)_\infty} \, A_k(q).
\mylabel{eq:MukAkid}        
\eeq
The function $A_k(q)$ was first studied by MacMahon \cite{MacM19} as
a generalization of
\beq
A_1(q) = \sum_{n=1}^\infty \sigma_1(n) q^n 
= \sum_{n=1} \frac{q^n}{(1-q^n)^2}.
\mylabel{eq:A1}
\eeq
He conjectured that the coefficients of $A_k(q)$ could be expressed
in terms of divisors functions. This conjecture was recently
proved by Andrews and Rose \cite{An-Ro} by showing that in general
$A_k(q)$ is a quasimodular form. The result also follows from
\eqn{MukAkid}, \eqn{symM2M} and the fact that the generating function
for $M_{2k}(n)$ is $P(q)$ times a quasimodular form, which was 
proved Atkin and the author \cite[Theorem 4.2]{At-Ga}. 
Andrews and Rose's proof is more direct. Andrews and Rose's 
were motivated by a certain curve-counting problem on Abelian
surfaces.


\section{Table}\label{sec:sptktab}

For reference we include a table of $\spt_k(n)$ for $1\le k \le 6$, $1\le n\le29$.
$$
\begin {array}{rrrrrrr} n{\backslash}k&1&2&3&4&5&6\\ \noalign{\smallskip}1&1&0
&0&0&0&0\\ \noalign{\smallskip}2&3&1&0&0&0&0\\ \noalign{\smallskip}3&5&5&1
&0&0&0\\ \noalign{\smallskip}4&10&15&7&1&0&0\\ \noalign{\smallskip}5&14&35
&28&9&1&0\\ \noalign{\smallskip}6&26&75&85&45&11&1\\ \noalign{\smallskip}7
&35&140&217&166&66&13\\ \noalign{\smallskip}8&57&259&497&505&287&91
\\ \noalign{\smallskip}9&80&435&1036&1341&1013&456\\ \noalign{\smallskip}
10&119&735&2044&3223&3081&1834\\ \noalign{\smallskip}11&161&1155&3787&
7149&8372&6293\\ \noalign{\smallskip}12&238&1841&6797&14916&20824&19125
\\ \noalign{\smallskip}13&315&2765&11648&29480&48192&52781
\\ \noalign{\smallskip}14&440&4200&19558&55902&105117&134643
\\ \noalign{\smallskip}15&589&6125&31703&101892&217945&321622
\\ \noalign{\smallskip}16&801&8975&50645&180245&433017&726650
\\ \noalign{\smallskip}17&1048&12731&78674&309297&828346&1564696
\\ \noalign{\smallskip}18&1407&18179&120932&518859&1534271&3231635
\\ \noalign{\smallskip}19&1820&25235&181664&849563&2759132&6432859
\\ \noalign{\smallskip}20&2399&35180&270600&1366441&4837638&12395504
\\ \noalign{\smallskip}21&3087&48055&395682&2154789&8283014&23195905
\\ \noalign{\smallskip}22&3998&65681&574329&3348972&13894554&42287433
\\ \noalign{\smallskip}23&5092&88299&820834&5119981&22856717&75274166
\\ \noalign{\smallskip}24&6545&118895&1166109&7733835&36968045&131143033
\\ \noalign{\smallskip}25&8263&157690&1634668&11520100&58818578&
223982780\\ \noalign{\smallskip}26&10486&209230&2279242&16985374&
92258215&375713010\\ \noalign{\smallskip}27&13165&274510&3142903&
24746334&142699970&619712403\\ \noalign{\smallskip}28&16562&359779&
4312063&35735413&218041302&1006599177\\ \noalign{\smallskip}29&20630&
466970&5859616&51073008&329162610&1611563058\end {array} 
$$


\noindent
\textbf{Acknowledgements}

\noindent
Firstly, I would like to thank Richard McIntosh for
showing me Alexander Patkowksi's paper \cite{Pa10}.
It was Patkowski's idea of using a limiting form
of Bailey's Lemma to derive spt-like results that first got me
started on the way to generalizing Andrews' spt-function.
Secondly, I would like to thank Mike Hirschhorn for hosting
my stay at UNSW in June 2010, 
and suggesting to me that I take a look at Stirling numbers of the second
kind. This was quite helpful in relating ordinary and symmetrized
moments. Finally, I would like to thank George Andrews,
Mikl\'os B\'ona,
Kathrin Bringmann, Jeremy Lovejoy and Karl Mahlburg
for their comments
and suggestions.
\bibliographystyle{amsplain}

\begin{thebibliography}{10}
\bibitem{An-CBMS-book}
G.~E.~Andrews,
\textit{{$q$}-Series: Their Development and Application in Analysis,
Number Theory, Combinatorics, Physics, and Computer Algebra},
CBMS Regional Conference Series in Mathematics,
66,
AMS,
Providence, R.I.,
1986.
\bibitem{An07a}
G.~E.~Andrews,
\textit{Partitions, Durfee symbols, and the {A}tkin-{G}arvan moments of ranks},
Invent. Math.
\textbf{169} 
(2007),
37--73.
\hfill\\
URL:\,\url{http://dx.doi.org/10.1007/s00222-007-0043-4}
\bibitem{An08b}
G.~E.~Andrews, 
\textit{The number of smallest parts in the partitions of {$n$}},
J. Reine Angew. Math.
\textbf{624} 
(2008),
133--142.
\hfill\\
URL:\,\url{http://dx.doi.org/10.1515/CRELLE.2008.083}
\bibitem{An-Ro}
G.~E.~Andrews and S.~C.~F.~Rose,
\textit{MacMahon's sum-of-divisors functions, Chebyshev polynomials,
and quasi-modular forms}, preprint.
\hfill\\
URL:\,\url{http://front.math.ucdavis.edu/1010.5769}
\bibitem{As-thesis}
M.~H.~Ashworth,
\textit{Congruence properties of coefficients of modular forms using sigma 
functions},
Ph.D. thesis, Oxford University, 1966.                          
\bibitem{At-Ga}
A.~O.~L.~Atkin and F.~G.~Garvan, 
\textit{Relations between the ranks and cranks of partitions},
Ramanujan J.
\textbf{7} 
(2003),
343--366.
\hfill\\
URL:\,\url{http://dx.doi.org/10.1023/A:1026219901284}
\bibitem{At-Sw}
A.~O.~L. Atkin and  P.~Swinnerton-Dyer,
\textit{Some properties of partitions},
Proc. London Math. Soc. (3)
\textbf{4} 
(1954),
84--106.
\hfill\\
URL:\,\url{http://dx.doi.org/10.1112/plms/s3-4.1.84}
\bibitem{Br-Ga-Ma}
K.~Bringmann, F.~Garvan and K.~Mahlburg,
\textit{Partition statistics and quasiharmonic {M}aass forms},
Int. Math. Res. Not. IMRN,
Vol.  2009, No. 1, 
63--97.
\hfill\\
URL:\,\url{http://dx.doi.org/10.1093/imrn/rnn124}
\bibitem{Br-Lo-Os}
K.~Bringmann, J.~Lovejoy and R.~Osburn,
\textit{Automorphic properties of generating functions for generalized 
rank moments and {D}urfee symbols},
Int. Math. Res. Not. IMRN
\textbf{YR 2010} 
238--260.
\hfill\\
URL:\,\url{http://imrn.oxfordjournals.org/content/2010/2/238.full.pdf#page=1&view=FitH}
\bibitem{Br-Ma-Rh}
K.~Bringmann, K.~Mahlburg and R.~Rhoades,
\textit{Asymptotics for rank and crank moments},
Bull. London Math. Soc., to appear.
\hfill\\
URL:\,\url{http://front.math.ucdavis.edu/0903.1297}
\bibitem{Br-Ma}
K.~Bringmann and K.~Mahlburg,                         
\textit{Inequalities between ranks and cranks},
Proc. Amer. Math. Soc.
\textbf{137} 
(2009),
2567--2574.
\hfill\\
URL:\,\url{http://dx.doi.org/10.1090/S0002-9939-09-09806-2}
\bibitem{Br08}
K.~Bringmann, 
\textit{On the explicit construction of higher deformations of
partition statistics},
Duke Math. J.
\textbf{144} 
(2008),
195--233.
\hfill\\
URL:\,\url{http://dx.doi.org/10.1215/00127094-2008-035}
\bibitem{Ek00}
A.~B.~Ekin, 
\textit{Some properties of partitions in terms of crank},
Trans. Amer. Math. Soc.
\textbf{352} 
(2000),
2145--2156.
\hfill\\
URL:\,\url{http://dx.doi.org/10.1090/S0002-9947-00-02306-0}
\bibitem{Fo-On}                  
A.~Folsom and K.~Ono,
\textit{The spt-function of Andrews},
Proc. Natl. Acad. Sci. USA 
\textbf{105}
(2008), 
20152--20156.
\hfill\\
URL:\,\url{http://mathcs.emory.edu/~ono/publications-cv/pdfs/111.pdf}
\bibitem{Ga-Ra-book}
G.~Gasper and M.~Rahman, 
\textit{Basic Hypergeometric Series},
Encycl. Math. Appl., 
Cambridge Univ. Press, Cambridge, 2004.

\bibitem{Ga88a}
F.~G.~Garvan,
\textit{New combinatorial interpretations of Ramanujan's partition congruences
mod $5$, $7$ and {$11$}},
Trans. Amer. Math. Soc.
\textbf{305}
(1988),
47--77.
\hfill\\
URL:\,\url{http://dx.doi.org/10.2307/2001040}
\bibitem{Ga10a}
F.~G.~Garvan,
\textit{Congruences for Andrews' smallest parts partition 
function and 
new congruences for Dyson's rank},
Int. J. Number Theory
\textbf{6}
(2010), 
1--29. 
\hfill\\
URL:\,\url{http://dx.doi.org/10.1142/S179304211000296X}
\bibitem{Ko76}
N.~Koblitz, 
\textit{{$2$}-adic and {$3$}-adic ordinals of {$(1/j)$}-expansion 
coefficients for the weight {$2$} {E}isenstein series},
Bull. London Math. Soc.
\textbf{9} 
(1977),
188--192.
\hfill\\
URL:\,\url{http://http://www.ams.org/leavingmsn?url=http://dx.doi.org/10.1112/blms/9.2.188}
\bibitem{Le92}
R.~Lewis,
\textit{Relations between the rank and the crank modulo {$9$}},
J. London Math. Soc. (2)
\textbf{45} 
(1992),
222--231.
\hfill\\
URL:\,\url{http://dx.doi.org/10.1016/0097-3165(92)90101-Y}
\bibitem{MacM19}
P.~A.~MacMahon,
\textit{Divisors of numbers and their continuations in the theory of 
partitions},
Proc. London Math. Soc. Ser.2 \textbf{19} (1921), 75--113. 
\hfill\\
URL:\,\url{http://dx.doi.org/10.1112/plms/s2-19.1.75}
\bibitem{On10}
K.~Ono,        
\textit{Congruences for the Andrews spt-function},
Proc. Natl. Acad. Sci. USA,
to appear.
\hfill\\
URL:\,\url{http://mathcs.emory.edu/~ono/publications-cv/pdfs/132.pdf}
\bibitem{Pa10}
A.~E.~Patkowski,
\textit{A strange partition theorem related to the second 
Atkin-Garvan
moment},
preprint.
\bibitem{Sl-OEIS}
N.~J.~A.~Sloane,
``The On-Line Encyclopedia of Integer Sequences'', published electronically,
2010.
\hfill\\
URL:\,\url{http://oeis.org}
\bibitem{Sl-A036969}
N.~J.~A.~Sloane,
\textit{Sequence \texttt{A036969}}, in
``The On-Line Encyclopedia of Integer Sequences'',
published electronically, 2010.
URL:\,\url{http://oeis.org/A036969}
\end{thebibliography}

\end{document}